\documentclass[11pt]{amsart}
\usepackage[utf8]{inputenc}
\usepackage[text={6.1in,8.5in},centering]{geometry}
\usepackage{amssymb,amsmath,amsthm,bm,fdsymbol}
\usepackage{tikz}
\usepackage[square,numbers]{natbib}

\usetikzlibrary{patterns}
\usetikzlibrary{fadings}

\newtheorem{thm}[equation]{Theorem}
\newtheorem*{thm*}{Theorem}
\newtheorem*{lem*}{Lemma}

\newtheorem{lem}[equation]{Lemma}

\newtheorem*{prop*}{Proposition}
\newtheorem{cor}[equation]{Corollary}

\theoremstyle{definition}
\newtheorem{defn}[equation]{Definition}

\numberwithin{equation}{section}

\newcommand{\U}{C}      
\newcommand{\W}{D}      

\newcommand{\Pmin}{\mathcal{P}_{min}}      

\title{Isoperimetric Formulas for Hyperbolic Animals}

\author[E.~Rold\'an]{\'Erika Rold\'an}
\address[E.~Rold\'an]{Zentrum Mathematik, TU M\"unchen, Garching b. M\"unchen, Germany}
\email{erika.roldan@ma.tum.de}

\author[R.~Toala-Enriquez]{Rosemberg Toala-Enriquez}
\address[R.~Toala-Enriquez]{Maths Learning Centre, De Montfort University}
\email{rosemberg.te@dmu.ac.uk}

\begin{document}
\begin{abstract}
An animal is a planar shape formed by attaching congruent regular polygons along their edges. In 1976, Harary and Harborth gave closed isoperimetric formulas for Euclidean animals. Here, we provide analogous formulas for hyperbolic animals. We do this by proving a connection between Sturmian words and the parameters of a discrete analogue of balls in the graph determined by hyperbolic tessellations. This reveals a complexity in hyperbolic animals that is not present in Euclidean animals.


\end{abstract}
\subjclass[2020]{00A69, 05A16, 05A20, 05B50, 52B60, 52C05,  05C07, 05C10, 52C20, 05D99}
\keywords{Extremal combinatorics, enumerative combinatorics, polyforms, polyominoes, polyiamonds, hexiamonds, animals, hyperbolic tessellations, isoperimetric inequalities.} 

\maketitle

\section{Introduction}
An \emph{animal} is a planar shape with a connected interior consisting of a finite number of tiles on a regular tessellation of the plane. Using Schläfli's notation, a $\{p,q\}$-tessellation consists of tiling the plane with regular $p$-gons with exactly $q$ of these polygons meeting at each vertex. We call an animal living in a $\{p,q\}$-tessellation a $\{p,q\}$-animal. The polygons forming an animal are called tiles. We define the perimeter to be the number of edges that are in the boundary of the animal. 
Let $\mathcal{P}^{p,q}_{min}(n)$ be the minimum perimeter that a $\{p,q\}$-animal with $n$ tiles can have.
In 1976 \cite{harary1976extremal}, Harary and Harborth gave the following isoperimetric formulas for Euclidean animals, that is, for all $\{p,q\}$-animals such that $(p-2)(q-2) = 4$.

    \begin{align*}
        \mathcal{P}^{3,6}_{min}(n) &= 2 \left\lceil \frac12( n+\sqrt{6n}) \right\rceil - n,  \\    
        \mathcal{P}_{min}^{4,4}(n) &= 2 \left\lceil 2\sqrt{n} \right\rceil,                  \\ 
        \mathcal{P}_{min}^{6,3}(n) &=  2 \left\lceil \sqrt{12n-3} \right\rceil.  
    \end{align*}

Harary and Harborth called animals attaining minimum perimeter \textit{extremal animals}. To our surprise, for almost 50 years, isoperimetric formulas for hyperbolic animals were unknown and in general extremal hyperbolic animals were unstudied. Recently \cite{MaRoTo}, Malen, Rold\'an and Toal\'a-Enr\'iquez gave a constructive way of generating a sequence of extremal hyperbolic animals that follows a spiral pattern---See Figure \ref{fig:spirals}. Based on these extremal spiral animals they were able to provide a recursive way of computing $\mathcal{P}^{p,q}_{min}(n)$ for hyperbolic animals.  Nevertheless, a closed formula for $\mathcal{P}^{p,q}_{min}(n)$ with $(p-2)(q-2)>4$ was still missing. Here, we provide these closed formulas. 

In our main result, we prove that for all $(p-2)(q-2) > 4$ and $n> p (q-2)$,   
\begin{align} \label{eqn:P_min}
            \Pmin^{p,q}(n) = 
                            \left( p-2-\frac{2}{\beta}\right) n +  \epsilon_{p,q}(n) , 
            \end{align}
where $\epsilon_{p,q}(n) $ is an error term that we explicitly define in Theorem \ref{thm:P_min}, Section \ref{sec:mainresults} and that is uniformly bounded for all $p,q,n$. The constant $\beta$ depends only on $p$ and $q$ and is given by
    \begin{equation}    \label{eqn:beta}
        \beta=\frac{(p-2)(q-2)+\sqrt{(p-2)^2(q-2)^2-4(p-2)(q-2)}}{2(p-2)}.
    \end{equation} 
For $n\leq p (q-2)$, the error term in Formula (\ref{eqn:P_min}) takes a different form. For completeness, we will also provide isoperimetric formulas for these values of $n$.  

\begin{figure}  
    \centering
    \includegraphics[scale=2.75]{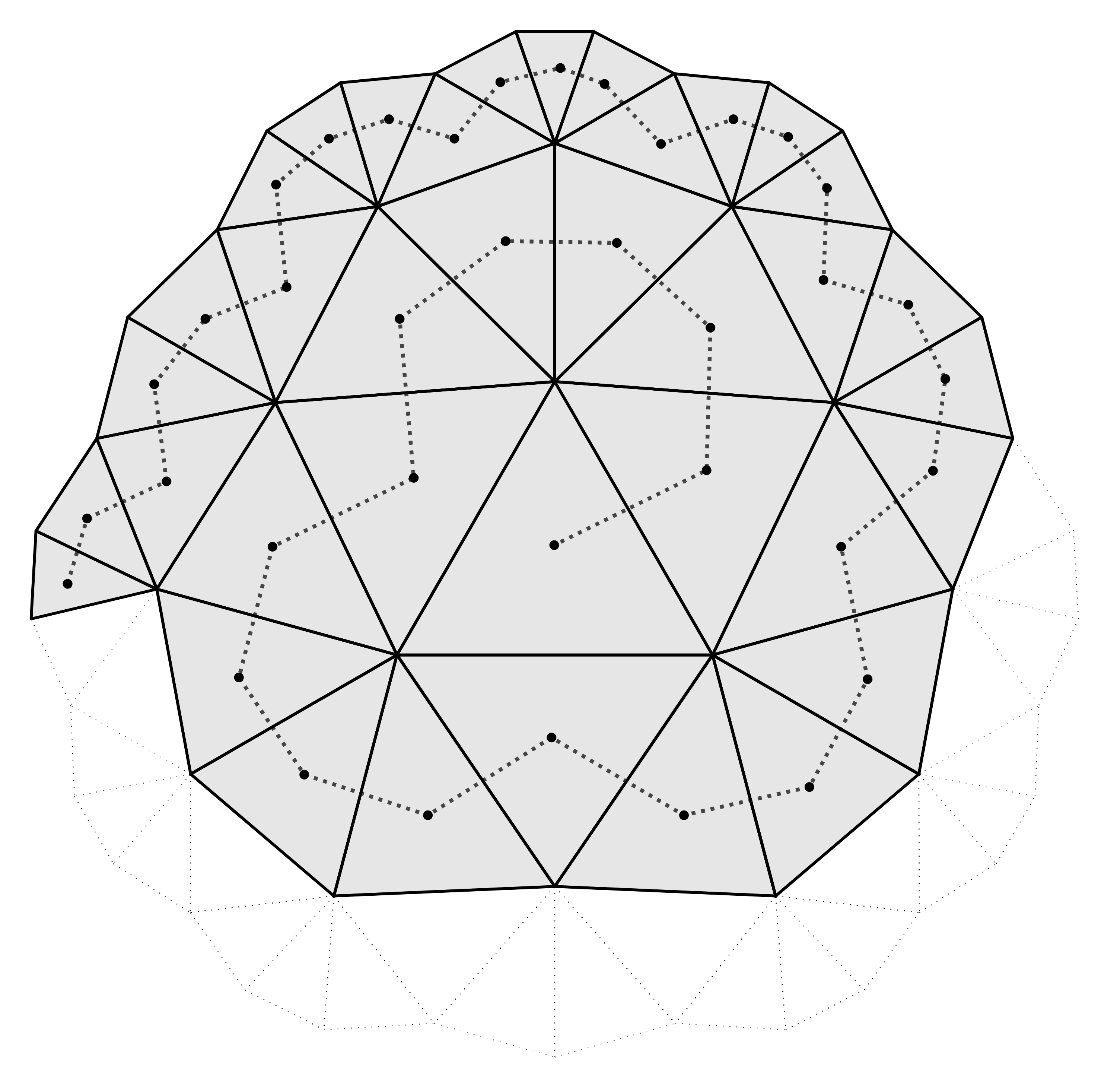}
    \includegraphics[scale=2.5]{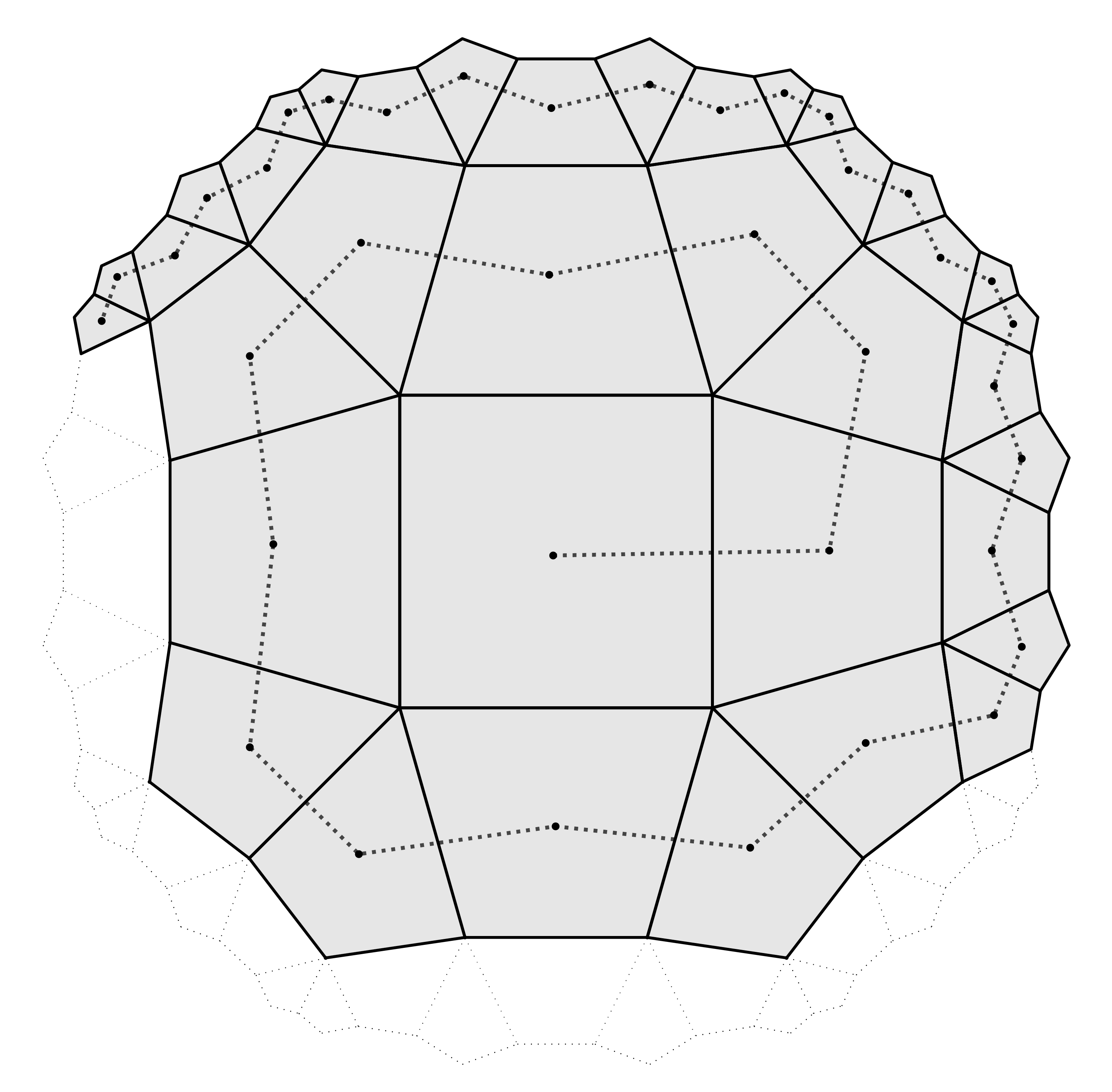}
    \includegraphics[scale=2.5]{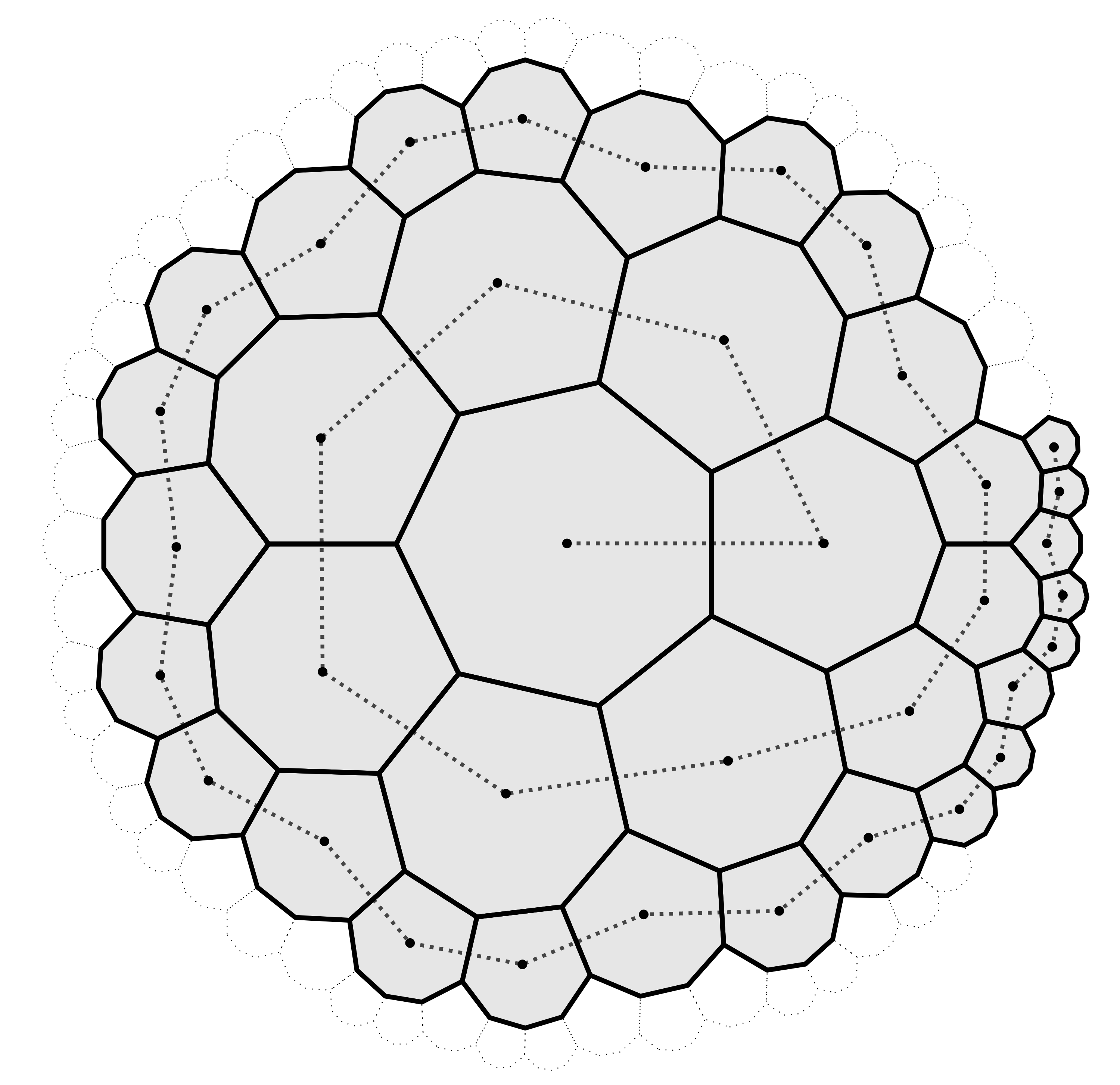}
    \caption{Examples of extremal hyperbolic animals \cite{MaRoTo}. Here we show spiral animals with 37 tiles for the hyperbolic cases $\{3,7\}, \{4,5\}$, and  $\{7,3\}$.}  
    \label{fig:spirals}
\end{figure}

\subsection{Main Results}\label{sec:mainresults}

For the rest of the paper, unless otherwise stated, we assume that $\{p,q\}$ corresponds to an hyperbolic tessellation, that is, $(p-2)(q-2)>4$. 

Let $\alpha$ be the solution of the quadratic equation $\alpha^2 +2\alpha + 1 = (p-2)(q-2)\alpha$ that fulfils $\alpha>1$.
The constants $\alpha$ and $\beta$ are related by 
    \begin{align}
    \beta = \frac{\alpha+1}{p-2}=\frac{q-2}{\alpha^{-1}+1}.       
    \end{align}

We define the sequences $\{n_k\}$ and $\{l_k\}$, for $k\geq 1$, as 
\begin{align}
        n_k &= 1 + \frac{p(q-2)}{\alpha-\alpha^{-1}} \left( \frac{\alpha^k-1}{\alpha-1} - \frac{\alpha^{-k}-1}{\alpha^{-1}-1}  \right), 
        \quad \text{and} \quad
        l_k=\frac{n_{k+1}-n_k}{p}.
        \label{eqn:n_k } 
    \end{align}

Now we have all the ingredients to state our main result.

    \begin{thm}    \label{thm:P_min}
 Let $(p-2)(q-2)>4$, $n>p(q-2)$, $k$ such that $n_k < n \le n_{k+1}$, \[ n' = n-n_k - \left \lfloor \frac{n-n_k-1 }{l_k} \right \rfloor \cdot l_k, \] and     \begin{align} \label{eqn:phi}
    \phi_k =  \begin{cases}
    \alpha^{-k}(\alpha-1)   & \text{if $p=3$ and } n'\leq \frac{l_k+l_{k+1} }{q-2} , \text{ or if $p=4$},\\    
    \alpha^{-k}(\alpha(\alpha-1)+1) ,   & \text{if $p=3$ and }  \frac{l_k+l_{k+1} }{q-2} < n' \leq l_k  , \\    
    \alpha^{-k}(\alpha-\beta)   & \text{if $p>4$}.  
  \end{cases}
    \end{align}
     If we define
        \begin{equation*}
        \epsilon_{p,q}(n)  = \frac{2n}{\beta} + 2
        \left \lfloor 1+\frac{1}{\beta} +
        \frac{p}{\alpha-1} +\frac{p\alpha^{-k}}{\alpha^{-1}-1} + \frac{\phi_k}{\beta} + \alpha^{-k} \left \lfloor \frac{n-n_k-1}{l_k} \right \rfloor  - \frac{n}{\beta} 
        \right \rfloor
        - 2   \left \lfloor \frac{n-n_k}{n_{k+1}-n_k}  \right \rfloor,
        \end{equation*}
        then
            \begin{align} \label{eqn:P_min2}
            \Pmin^{p,q}(n) = \left( p-2-\frac{2}{\beta}\right) n +  \epsilon_{p,q}(n). 
            \end{align}
\end{thm}

Using the definition of $n_k$, the condition $n_k < n \le n_{k+1}$ is equivalent to 
\begin{align}
    k=\left \lfloor \log_{\alpha} \left( \frac{\Delta + \sqrt{\Delta^2-4\alpha}}{2} \right) \right \rfloor, \quad \text{where} \quad   \Delta=\alpha+1+\frac{(n-2) (\alpha-1)(\alpha-\alpha^{-1})}{p(q-2)}. \label{eqn:k}        
    \end{align}
This gives a formula for computing $k$ for any $n\geq 1$. Thus, equation (\ref{eqn:P_min2}) is a closed formula for $\Pmin^{p,q}(n)$ valid for $n>p(q-2)$. For $n\leq p (q-2)$ the values of $\Pmin^{p,q}(n)$  can be obtained by following the spiral construction of extremal animals defined in \cite{MaRoTo}, and analyzing the increment (of either $p-2$ or $p-4$) when a new tile is attached. This results in the following formulas
    \begin{align*}
        \Pmin^{p,q}(n) = 
        \begin{cases}
              p+(p-2)(n-1), & \text{ if } 1\le n < q, \\
              p + (p-2)\left(n-1- \left \lfloor \frac{n-2}{q-2}\right \rfloor \right) + (p-4) \left \lfloor \frac{n-2}{q-2}\right \rfloor, & \text{ if } q\le n < p(q-2) , \\
              p(p-3) + (q-3)(p-2)p, & \text{ if } n = p(q-2). \\
        \end{cases}
    \end{align*}

It is not clear from Theorem \ref{thm:P_min} what is the asymptotic behaviour of $\Pmin^{p,q}(n)$. To accomplish this, we prove in the following Theorem that $\epsilon_{p,q}(n)$ is uniformly bounded by a constant.

\begin{thm} \label{thm:bounds}
Let $n>p(q-2)$ and $\epsilon_{p,q}(n)$ be defined as above. Then, $0< \epsilon_{p,q}(n)<22 $. Moreover, for $n>n_4$,
    \begin{align}   \label{eqn:bound2}
      \left| \epsilon_{p,q}(n) - 2\left(1+\frac{1}{\beta} + \frac{p}{\alpha-1} \right)  \right| < 2.6.  
    \end{align}
      
\end{thm}

\begin{cor}
Let $(p-2)(q-2)>4$, then
\begin{align} 
            \Pmin^{p,q}(n) = 
                            \left( p-2-\frac{2}{\beta}\right) n +  o(1) . 
            \end{align}

\end{cor}

The rest of the paper is structured as follows: in Section \ref{section:Sturmian}, we describe how Sturmian words arise from the study of the degrees (i.e., the number of edges adjacent to a vertex) of the perimeter vertices of a certain sequence of extremal hyperbolic animals. In the same section, in Theorem \ref{thm:formula-W}, we give a closed formula for the values of these degrees. In Section \ref{section:P_min}, we use this formula to prove Theorem \ref{thm:P_min} and Theorem \ref{thm:bounds}. 
For clarity of exposition, we have postponed the proof of Theorem \ref{thm:P_min} and some of the painful algebraic computations needed along the paper to Sections  \ref{section:formula-W} and \ref{section:Pain}, respectively.

\section{Extremal hyperbolic animals and Sturmian words} \label{section:Sturmian}

We use the following construction of sequences of extremal hyperbolic animals introduced in \cite{MaRoTo}.

\begin{defn} \label{defn:A_p,q}
Denote by $A_{p,q}(1)$ the animal with only one $p$-gon. Then, construct $A_{p,q}(k)$ from $A_{p,q}(k-1)$ by adding precisely the tiles needed so that all the perimeter vertices of $A_{p,q}(k-1)$ get surrounded by $q$ tiles. We call $A_{p,q}(k)$ the \emph{complete $k$-layered $\{p,q\}$-animal}. Let $d_k$ be the word obtained by concatenating the degrees of the perimeter vertices of $A_{p,q}(k)$. Denote by $n_k$ and $P_k$ the number of tiles and the perimeter of $A_{p,q}(k)$, respectively.\footnote{In Theorem 3.1 of \cite{MaRoTo} it was proven that $n_k$ satisfies Equation (\ref{eqn:n_k }).}
\end{defn}


In what follows we use the abbreviation $2^a$ to describe a string of $a$ consecutive 2’s. The expressions $3^a$, $4^a$, etc., are interpreted similarly. If $a= 0$ then we have an empty block.

On the one hand, in \cite{MaRoTo}, the words $d_{k}$ were computed recursively in the following manner: $d_{1}$ is the word $2^p$; then, $d_{k}$ is constructed from $d_{k-1}$ by replacing each element of $d_{k-1}$ with a string according to the rules contained in Table \ref{table:d_k}. Figure \ref{fig:A3-recursion} depicts the geometric reasoning behind these substitution rules. 

\begin{table}
\caption{Table showing substitution rules for words $d_k$.}
\label{table:d_k}
    \centering
\begin{tabular} {  c|c|c|c}
Cases & $p=3$ and $q\ge 7$ & $p\ge 4$and $q \ge 4$  & $p \ge 7$ and $q=3$ \\
 \hline
 \hline
 
 &   $2 \to 4 3^{q-4}$  &  $2 \to 3 2^{p-4}\left( 3 2^{p-3} \right)^{q-3}$ &  $2 \to 3 2^{p-4}$     \\
Substitution &  $3 \to 43^{q-5}$          
	& $3  \to 32^{p-4}\left( 3 2^{p-3} \right)^{q-4}$      
	&  $32 \to  3 2^{p-5}$        \\
&  $4\to 43^{q-6}$         
    & &      \\
    \hline
\end{tabular}
\end{table}

\begin{figure}[h]  
    \centering
   \includegraphics[scale=4.5]{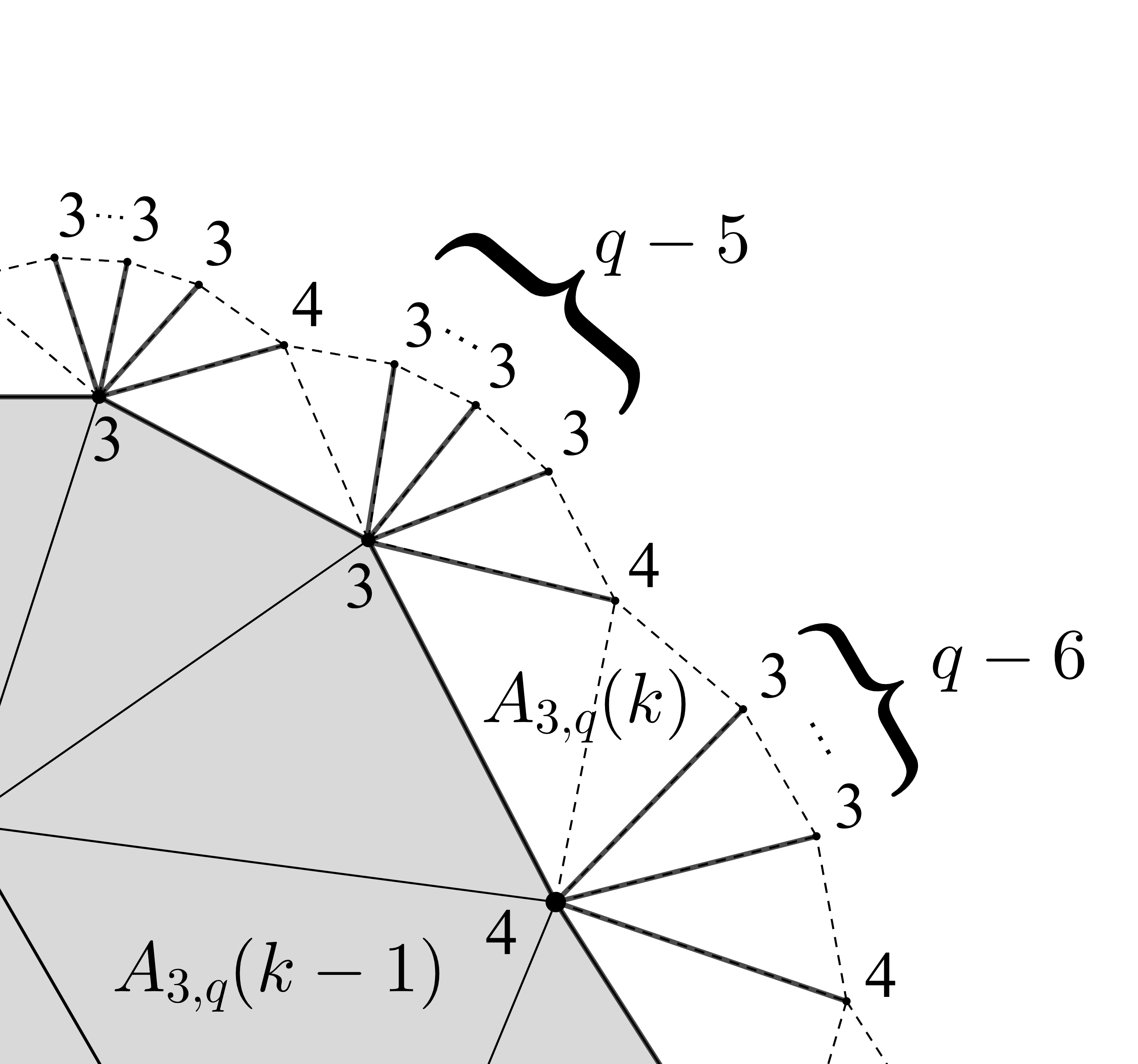}
   \includegraphics[scale=0.045]{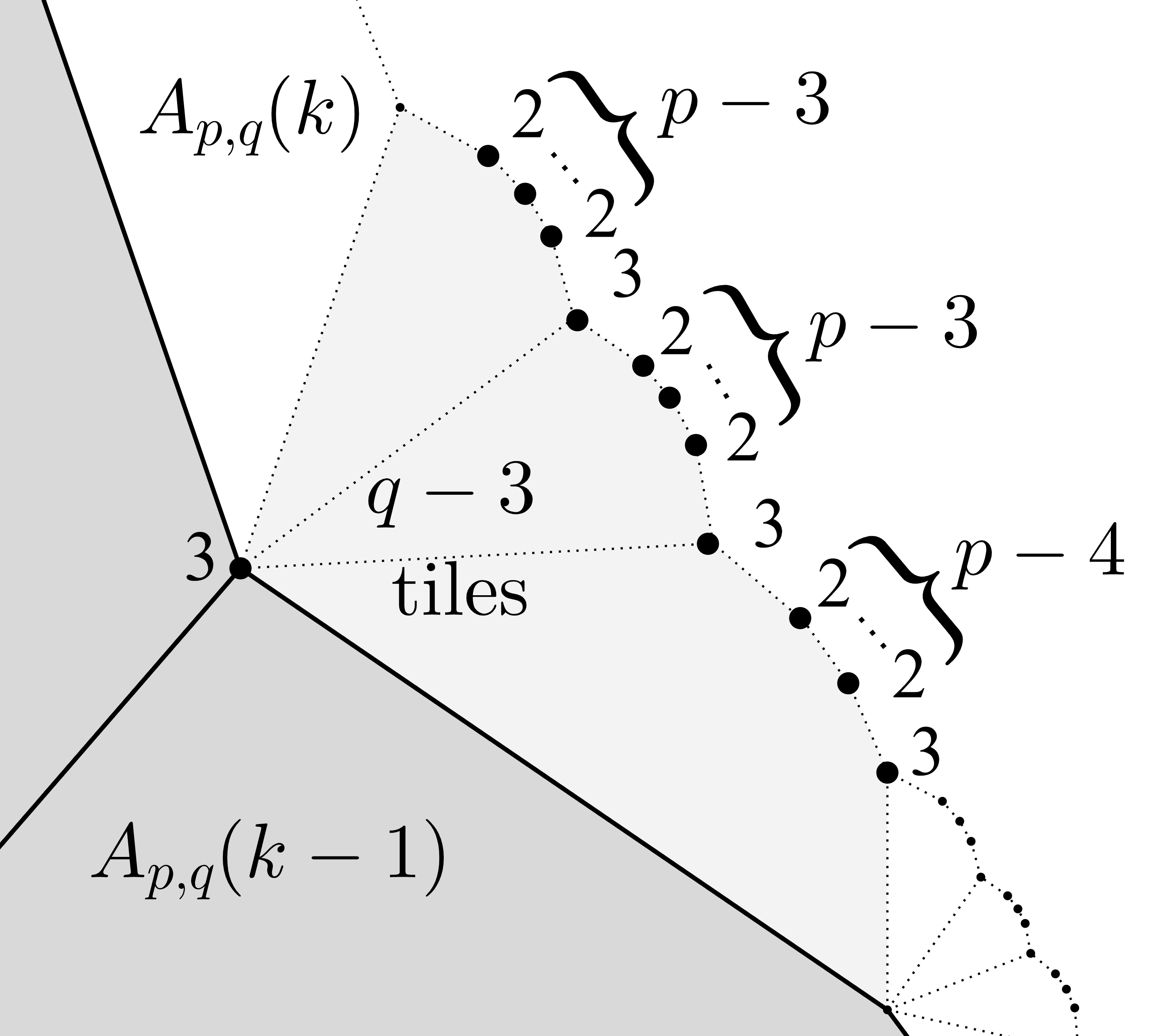}
\includegraphics[scale=6]{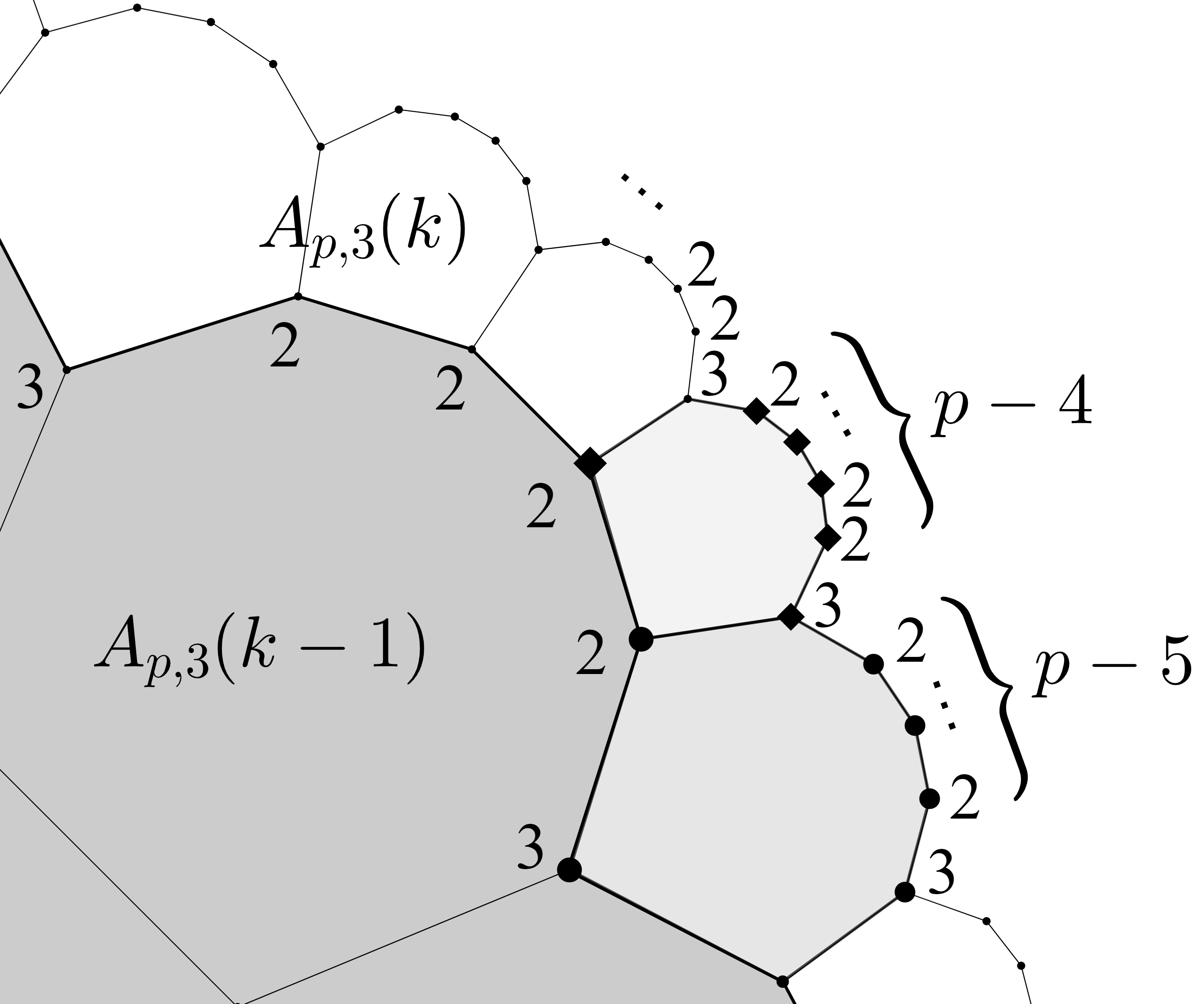}
       \caption{\textit{Left:} When $p=3$, each vertex of degree $d=3$ or $4$ in $A_{p,q}(k-1)$ contributes to a vertex of degree 4 in $A_{p,q}(k)$ and $q-d-2$ vertices of degree equal to $3$. \textit{Center:} When $p\geq4$, each vertex of degree equal to $3$ in $A_{p,q}(k-1)$ contributes to $q-3$ vertices of degree $3$ in $A_{p,q}(k)$ and  $(q-3)(p-3)-1$ vertices of degree $2$. \textit{Right:} When $q=3$, each vertex of degree $3$ in $A_{p,q}(k-1)$ is paired up with a vertex of degree $2$ they together contribute to a vertex of degree $3$ and $p-5$ vertices of degree $2$ in $A_{p,q}(k)$. The other vertices of degree $2$ in $A_{p,q}(k-1)$ contribute to a vertex of degree $3$ and $p-4$ vertices of degree $2$ in $A_{p,q}(k)$.}  \label{fig:A3-recursion}
\end{figure}

On the other hand, the substitution rules of Table \ref{table:d_k}  can be rephrased if we define the word sequences $U_k$ and $W_k$ as in Table \ref{table:UW}. Then $d_k=(U_k)^p$ if $p>3$, and $k\ge 1$. If $p=3$, and $k\ge 2$, then $d_k=(U_k U_{k-1})^3$. \\ 

\begin{table}
\caption{Table showing equations for recurrences of words $U_k$ and $W_k$.}
\label{table:UW}
    \centering
\begin{tabular} {  c|c|c|c}
Cases & $p=3$ and $q\ge 7$ & $p\ge4$ and $q \ge 4$  & $p \ge 7$ and $q=3$ \\
 \hline
 \hline
	&   $U_1  =3,  \, W_1=4 $         
	&  $U_1=2, \,  W_1=3$ 
	&  $U_1 = 2, \, W_1 = 32  $  \\
	   Recurrence
	&  $ U_{k+1} = W_k U_k^{q-5}$         
	&  $U_{k+1} = W_k U_k^{p-4} \left( W_k U_k^{p-3} \right)^{q-3}$
	&  $U_{k+1} = W_k U_k^{p-5}$       \\
	
	&  	$W_{k+1} = W_k U_k^{q-6}$       
	&   $ W_{k+1} = W_k U_k^{p-4} \left( W_k U_k^{p-3} \right)^{q-4} $
	&  	$W_{k+1} = W_k U_k^{p-6} $\\
	\hline
  \end{tabular}
\end{table}

We will use the theory of Sturmian words to find closed formulas for the $i$-th element of $W_k$, $U_k$ and $d_k$. To do so, we relate the Sturmian word of $\beta$ with the limit word of the nested sequence of words $W_k$, that we represent by $W=\lim_{k\to \infty} W_k$.
\begin{defn} \label{defn:Sturmian}
 Let $\beta$ be as in (\ref{eqn:beta}), define the sequence
	$$B(i) = \lfloor (i+1)\beta \rfloor - \lfloor i \beta \rfloor - \lfloor \beta \rfloor$$
for $i\geq 1$. Let $B$ be the word obtained by concatenating the elements of the sequence $\{B(i)\}_{i\ge 1}$. $B$ is known as the  \textit{Sturmian word} of $\beta$.
\end{defn}

 
Shallit proved in 1991 \cite{shallit1991characteristic}, that $B$ can also be constructed in the following recursive way: let $[b_0; b_1, b_2, b_3, \dots]$ be the continued fraction expansion of $\beta$, then $B$ is equal to $\lim_{k\to \infty} B_{k}$, where, 
	\begin{align*}
	B_0&=0	, \\
	B_1&=0^{b_1-1}1	, \\
	B_{k}&=B_{k-1}^{b_k}B_{k-2}, n\geq 2.
	\end{align*}

Using these two characterizations of the Sturmian word of $\beta$, we prove in Section \ref{section:Pain} the following Theorem that gives a closed formula for the elements of $W$.

\begin{thm} \label{thm:formula-W}
Let $\beta$ be as in (\ref{eqn:beta}). The $i$-th element of  $W=\lim_{k\to \infty}W_k$ is equal to
    \begin{align}   \label{eqn:formula-W}
    q - \lfloor i \beta \rfloor + \lfloor (i-1) \beta \rfloor      .  
    \end{align}
\end{thm}   

\section{Proof of Theorem \ref{thm:P_min} and Theorem \ref{thm:bounds}}      \label{section:P_min}



We remind the reader that sequence of animals $A_{p,q}(k)$, that was defined in Definition \ref{defn:A_p,q}, has $n_k$ tiles with
    \begin{equation}
        n_k = 1 + \frac{p(q-2)}{\alpha-\alpha^{-1}} \left( \frac{\alpha^k-1}{\alpha-1} - \frac{\alpha^{-k}-1}{\alpha^{-1}-1}  \right) . \label{eqn:n_k}
        \end{equation}
    In \cite{MaRoTo}, it was proved that  $A_{p,q}(k)$ is extremal and its perimeter is
        \begin{equation}
        \Pmin^{p,q}(n_k) = \frac{p}{\alpha-\alpha^{-1}} \left( \alpha^k - \alpha^{-k} + \alpha^{k-1} - \alpha^{-(k-1)}\right). \label{eqn:P_k}
    \end{equation}
To streamline the notation, in what follows we denote $\Pmin^{p,q}(n_k)$ by $P_k$. Using Equations (\ref{eqn:P_k}) and (\ref{eqn:n_k}) we prove the following Lemma that we use as an essential ingredient for proving Theorem \ref{thm:P_min}.

\begin{lem} \label{lem:P_k-n_k}
Let $k\ge 1$, the perimeter $P_k$ and number of tiles $n_k$ of $A_{p,q}(k)$ satisfy the following equation
\begin{align}   \label{eqn:P_k-n_k}
    P_k- \left( p-2 \right) n_k = 2\left( \frac{p}{\alpha-1} + \frac{p\cdot \alpha^{-k}}{\alpha^{-1} -1 } +1 +\frac{1}{\beta} - \frac{n_k}{\beta}\right).
    \end{align}
\end{lem}

\begin{proof}
This follows from Equations (\ref{eqn:n_k}) and (\ref{eqn:P_k}). The details of these calculations can be found in Section \ref{section:Pain}. 
\end{proof}


Let $n>1$ and consider $k$ such that $n_k< n\le n_{k+1}$. Let $m=m(n)$ be the integer number greater or equal than $0$ defined by the condition
    \begin{align} \label{eqn:formula-m}
        \sum_{i=1}^m (q-d_{k}(i)) \leq n-n_k-1 < \sum_{i=1}^{m+1} (q-d_{k}(i)),
    \end{align}
where $d_{k}(i)$ denotes the $i$-th element of the word $d_k$. The sum on the left is interpreted as $0$ for $m=0$.

Here, we prove the following closed formula for $m$ in terms of $n$.

\begin{lem}     \label{lem:m}
Consider $k\ge 2$, and $n$ such that $n_k < n \le n_{k+1}$. Let $l_k = \dfrac{n_{k+1}-n_k}{p}$, and $
    n' = n-n_k - \left \lfloor \frac{n-n_k-1 }{l_k} \right \rfloor \cdot l_k $. 
If $m$ satisfies Equation (\ref{eqn:formula-m}), then
    \begin{align}  \label{eqn:m}
    m(n) = \left \lfloor \frac{n- n_k - \phi_k }{\beta}   - \left \lfloor \frac{n-n_k-1 }{l_k} \right \rfloor \cdot \alpha^{-k} \right \rfloor +  \left \lfloor \frac{n- n_k }{n_{k+1}-n_k} \right \rfloor .
    \end{align}
    Where   $\phi_k$ is given by equation (\ref{eqn:phi}).
\end{lem}

\begin{proof}
Theorem \ref{thm:formula-W} give us a formula for the $i$-th element of $W$, with this we get a formula for the $i$-th element of $U_k$. Then it is possible to invert condition (\ref{eqn:formula-m}), thanks to the structure of $d_k$ in terms of $U_k$. The details of the calculations can be found in Section \ref{section:Pain}. 
\end{proof}

Now we have all the ingredients to prove Theorem \ref{thm:P_min}. 


    \begin{proof}[Proof of Theorem \ref{thm:P_min}]
 In \cite{MaRoTo}, it is shown that
    \begin{align}   \label{eqn:Pmin-m}
    \Pmin^{p,q}(n) &= (p-2)(n-n_k) + P_k-2m .
    \end{align}
Simplifying the expression in Equation (\ref{eqn:Pmin-m}) using Lemmas \ref{lem:P_k-n_k} and \ref{lem:m} we obtain the expression for $\Pmin^{p,q}(n)$ given in Equation (\ref{eqn:P_min2}).
    \end{proof}

The rest of this section is devoted to the proof of Theorem \ref{thm:bounds}. 

\begin{lem} \label{lem:bounds}
For all $p,q$, the following inequalities hold
\begin{align}
    \frac{1}{\alpha} < 0.4, \quad \frac{1}{\beta} < 1.4, \quad \frac{p}{\alpha-1}< 4.4,  \quad \frac{\phi_2}{\beta} < 2.8 , \quad \text{and} \quad \frac{p}{\alpha^2} < 1.1.
\end{align}
\end{lem}
\begin{proof}
We claim that the maxima of $\frac{1}{\alpha}, \frac{1}{\beta},  \frac{p}{\alpha-1},  \frac{\phi_2}{\beta} ,  \frac{p}{\alpha^2}$ are attained at the smallest values of $p,q$. Then the bounds are found by evaluating at the cases $\{p,q\}=\{3,7\}, \{4,5\}, \{5,4\}, \{7,3\}$. Details can be found in Section \ref{section:Pain}.

\end{proof}

\begin{proof}[Proof of Theorem \ref{thm:bounds}]
Now we prove the bounds on $\epsilon_{p,q}(n)$. Recall that $\alpha>1$ and $k\ge 2$, note that 
    \begin{align*}
        \frac{\epsilon_{p,q}(n)  }{2}  & \le  \frac{n}{\beta} +  
         1+\frac{1}{\beta} +
        \frac{p}{\alpha-1} +\frac{p\alpha^{-k}}{\alpha^{-1}-1} + \frac{\phi_k}{\beta} + \alpha^{-k} \left \lfloor \frac{n-n_k-1}{l_k} \right \rfloor  - \frac{n}{\beta} \\
        & \le  1+\frac{1}{\beta} +
        \frac{p}{\alpha-1} + \frac{\phi_k}{\beta} + \alpha^{-k} p \\
        & \le   1+\frac{1}{\beta} +
        \frac{p}{\alpha-1} + \frac{\phi_2}{\beta} + \alpha^{-2} p  .
    \end{align*}
Where, we have used $\frac{p}{\alpha^{-1}-1}<0$ and  $\left \lfloor \frac{n-n_k-1}{l_k} \right \rfloor \le p$. Using Lemma \ref{lem:bounds}, we conclude 
    \begin{align*}
      \epsilon_{p,q}(n)  \le  2\left( 1+\frac{1}{\beta} +
        \frac{p}{\alpha-1} + \frac{\phi_2}{\beta} + \alpha^{-2} p  \right)  <2\left( 1+1.4+4.4+2.8+1.1\right) < 22 .
    \end{align*}

Now, using Equation (\ref{eqn:P_k-n_k}), we know
    \begin{align*}
        \epsilon_{p,q}(n_{k+1}) &= \Pmin^{p,q}(n_{k+1}) - \left(p-2-\frac{2}{\beta} \right) n_{k+1}  \\
            & = 2\left(\frac{p}{\alpha-1} + \frac{p \alpha^{-(k+1)}}{\alpha^{-1}-1}  + 1 + \frac{1}{\beta}\right) \ge 2\left( 1 + \frac{1}{\beta}\right) .
    \end{align*}
And for, $n_k<n<n_{k+1}$, we get
    \begin{align*}
    \frac{\epsilon_{p,q}(n)}{2}  & \ge  \frac{n}{\beta} +  \left(
         1+\frac{1}{\beta} +
        \frac{p}{\alpha-1} +\frac{p\alpha^{-k}}{\alpha^{-1}-1} + \frac{\phi_k}{\beta} + \alpha^{-k} \left \lfloor \frac{n-n_k-1}{l_k} \right \rfloor  - \frac{n}{\beta} -1 \right)         \\
        & \ge \frac{1}{\beta} + \frac{p}{\alpha-1}  +\frac{p\alpha^{-k}}{\alpha^{-1}-1}  = \frac{1}{\beta} + \frac{p}{\alpha-1}  - \frac{p\alpha^{-(k-1)}}{\alpha-1}
        \ge \frac{1}{\beta} .
    \end{align*} 
Thus, $\epsilon_{p,q}(n) >0$ for all $n$. 

Now we prove bound (\ref{eqn:bound2}). For $n=n_{k+1}$, and $k\ge 2$,
      \begin{align*}
        \left | \epsilon_{p,q}(n_{k+1}) - 2\left(1+ \frac{p}{\alpha-1} + \frac{1}{\beta}\right) \right| & = \left|  \frac{2 p \alpha^{-(k+1)}}{\alpha^{-1}-1} \right| =  \frac{2 p \alpha^{-k}}{\alpha-1}      \\
        & \le   \frac{2 p \alpha^{-2}}{\alpha-1} = \frac{2 p}{\alpha-1} \cdot \frac{1}{\alpha^2} <  2\cdot 4.4\cdot (0.4)^2 < 1.5  .
    \end{align*}
For $n_k<n<n_{k+1}$ and $k\ge 4$,
          \begin{align*}
      -\frac{2p\alpha^{-(k-1)}}{\alpha-1}  -2 & <   \epsilon_{p,q}(n_{k+1}) - 2\left(1+ \frac{p}{\alpha-1} + \frac{1}{\beta}\right) < \frac{2\phi_k}{\beta}+\frac{2p}{\alpha^k}   \\
    -\frac{2p\alpha^{-3}}{\alpha-1}  -2 & <   \epsilon_{p,q}(n_{k+1}) - 2\left(1+ \frac{p}{\alpha-1} + \frac{1}{\beta}\right) < \frac{2\phi_2}{\alpha^2 \beta}+\frac{2p}{\alpha^4}   \\
    -1.5\cdot 0.4  - 2 & <   \epsilon_{p,q}(n_{k+1}) - 2\left(1+ \frac{p}{\alpha-1} + \frac{1}{\beta}\right) < 2\cdot 2.8(0.4)^2 + 2\cdot  1.1(0.4)^2  \\
     - 2.6 &<   \epsilon_{p,q}(n_{k+1}) - 2\left(1+ \frac{p}{\alpha-1} + \frac{1}{\beta}\right) <  1.3 .
    \end{align*}
This concludes the proof of Theorem \ref{thm:bounds}.

\end{proof}

\section{Proof of Theorem \ref{thm:formula-W}}  \label{section:formula-W}

To prove Theorem \ref{thm:formula-W}, we begin by finding the continued fraction expansion of $\beta$. Then we use Shallit's Theorem to prove that the Sturmian word of $\beta$ and the word $W$ are essentially the same. We do this in Lemmas \ref{lem:beta} and \ref{lem:Shallit}, respectively.

\begin{lem}     \label{lem:beta}
The continued fraction expansion of $\beta$ is 
    \begin{align*}
    &    [q-4;\overline{1,q-6}], && \text{ if $p=3$ and $q\ge 7$,}      \\
    &    [q-3;\overline{2,q-4}], && \text{ if $p=4$ and $q\ge 5$,}      \\
    &    [q-3;\overline{1,p-4,1,q-4}], && \text{ if } p,q>4,        \\
    &    [1;1,\overline{p-4,2}], && \text{ if $p\ge 5$ and $q=4$,}      \\
    &    [0;1,p-5,\overline{1,p-6}], && \text{ if $p\ge 7$ and $q= 3$.}        
    \end{align*}
\end{lem}

\begin{proof}
$\beta$ satisfies the quadratic equation  $(p-2)\beta^2-(p-2)(q-2)\beta+q-2=0$, then $(p-2)\beta^2=(p-2)(q-2)\beta - (q-2)$. If $p,q>4$, note that $0< \frac{q-2}{(p-2)\beta}<1$, so
	\begin{align*}
	\beta &= q-2 - \frac{q-2}{(p-2)\beta} = q-3 + \frac{(p-2)\beta - (q-2)}{(p-2)\beta}	\\
			&= q-3 +\frac{1}{ \frac{(p-2)\beta}{(p-2)\beta - (q-2)} } 
			        = q-3 +\frac{1}{1 + \frac{q-2}{(p-2)\beta - (q-2)} } \\ 
	        &= q-3 +\frac{1}{ 1+ \frac{1}{ \frac{(p-2)\beta - (q-2)}{q-2} } } 
			         = q-3 +\frac{1}{1+ \frac{1}{ p-3 - \frac{1}{\beta} } }
			          \\ 
			& = q-3 +\frac{1}{1+ \frac{1}{ p-4 + \frac{\beta - 1}{\beta} } } = q-3 +\frac{1}{1+ \frac{1}{ p-4 + \frac{1}{ \frac{\beta}{\beta-1} } } } 
			     \\
			& = q-3 +\frac{1}{1+ \frac{1}{ p-4 + \frac{1}{1 + \frac{1}{\beta-1} } } } =  q-3 +\frac{1}{1+ \frac{1}{ p-4 + \frac{1}{1 + \frac{1}{q-4+ \dots  } } } } \\    
	\end{align*}
By continuing in this way we obtain $\beta = [q-3;\overline{1,p-4,1,q-4}]$ for $p,q>4$. The remaining cases are analogous. 
\end{proof} 

Recall that $B$ is the Sturmian word of $\beta$---See Definition \ref{defn:Sturmian}. The following result gives us a recursive way of describing $B$ through the continued fraction expansion of $\beta$.

\begin{thm} [Shallit \cite{shallit1991characteristic}] \label{thm:Shallit}
Let $\beta$ be an irrational number with continued fraction expansion $[b_0; b_1, b_2, b_3, \dots]$. Define 
	\begin{align*}
	B_0&=0	, \\
	B_1&=0^{b_1-1}1	, \\
	B_{k}&=B_{k-1}^{b_k}B_{k-2},\quad \text{for} \quad n\geq 2 .
	\end{align*}
Then $B=\lim_{k\to \infty} B_{k}$. 
\end{thm}

We now use Shallit's Theorem to show that $W$ and $B$ are closely related. To make the relation between $W$ and $B$ more transparent we introduce the words $\U_k, \W_k$ defined by the recurrence equations contained in Table \ref{table:CD}. Note that these are the same recurrence relations as $U_k$ and $W_k$, that we defined in Table \ref{table:UW}, but starting at 1 and 0, respectively. 

\begin{table}
\caption{Table showing equations for recurrences of words $C_k$ and $D_k$.}
\label{table:CD}
    \centering
\begin{tabular} {  c|c|c|c}
Cases & $p=3$ and $q\ge 7$ & $p\ge4$ and $q \ge 4$  & $p \ge 7$ and $q=3$ \\
 \hline
 \hline

	&  $\U_1  =1,  \, \W_1=0$            
	& $\U_1=1, \,  \W_1=0$       
	& $\U_1 = 1, \, \W_1 = 01$   \\
	 Recurrence &
	  $\U_{k+1} = \W_k \U_k^{q-5}$        
	& $\U_{k+1} = \W_k \U_k^{p-4} \left( \W_k \U_k^{p-3} \right)^{q-3}$
	& $\U_{k+1} = \W_k \U_k^{p-5}$      \\
	& 	$\W_{k+1} = \W_k \U_k^{q-6}$       
	&  $\W_{k+1} = \W_k \U_k^{p-4} \left( \W_k \U_k^{p-3} \right)^{q-4} $
	& 	$\W_{k+1} = \W_k \U_k^{p-6} $  \\
	\hline
\end{tabular}	
\end{table}	
	
Let $\W=\lim_{k\to \infty} \W_k$. Observe that $\W(i)=4-W(i)$, if $p=3$, and $\W(i)=3-W(i)$, if $p>3$.

\begin{lem} \label{lem:Shallit} 
Denote by $\W(i)$ and $B(i)$ the $i$-th elements of the words $\W$ and $B$, respectively. Then $\W(i)=B(i-1)$. 
\end{lem}

\begin{proof}
Denote by $\W_k'$ the word $\W_k$ without its first (left-most) digit. That digit is always 0, i.e., $\W_k=0\W_k'$.

Fix $p,q>4$. We claim that, for $k\ge 1$,
    \begin{align}
        0B_{4k-3} \W_{k}' &= \W_k\U_{k} ,     \label{eqn:SW1-pq}   \\
        B_{4k-2} &= \W_{k}' \U_k^{p-4} 0 ,     \label{eqn:SW2-pq}        \\
        0B_{4k-1} \W_{k+1}' &= \U_{k+1} ,     \label{eqn:SW3-pq}   \\
        B_{4k} &= \W_{k+1}' 0 .     \label{eqn:SW4-pq} 
    \end{align}
We prove these relations by induction. The continued fraction expansion of $\beta$ for $p,q>4$ is  $[q-3;\overline{1,p-4,1,q-4}]$. Then, Shallit's recurrence relations, stated in Theorem \ref{thm:Shallit}, for $k\ge 1$, give
    \begin{align*}
        B_{4k+1} &= B_{4k} B_{4k-1}  , \\
        B_{4k+2} &= B_{4k+1}^{p-4} B_{4k} .\\
        B_{4k+3} &= B_{4k+2} B_{4k+1}  , \\
        B_{4k+4} &= B_{4k+3}^{q-4} B_{4k+2} .
    \end{align*}
Also, $B_1=1$,  $B_2= 1^{p-4}0$, $B_3=1^{p-4}01$, $B_4=1^{p-4}\left(01^{p-3}\right)^{q-4}0$, and $\W_2=01^{p-4}\left(01^{p-3}\right)^{q-4}$, $\U_2=01^{p-4}\left(01^{p-3}\right)^{q-3}$ . So, equations (\ref{eqn:SW1-pq})-(\ref{eqn:SW4-pq}) hold for $k=1$. Assume these equations are also true for $k>1$. Then, 
    \begin{align*}
    0B_{4k+1} \W_{k+1}'&= 0 B_{4k}B_{4k-1}\W_{k+1}'     \\
                &= 0 \W'_{k+1}0 B_{4k-1} \W_{k+1}'       \\
                &= \W_{k+1} \U_{k+1} .
     \end{align*}
Similarly, 
    \begin{align*}
        B_{4k+2} &= B_{4k+1}^{p-4}B_{4k} = \left( B_{4k}B_{4k-1}\right)^{p-4} B_{4k}      \\
                &= \left( \W_{k+1}'0B_{4k-1}    \right)^{p-4} \W_{k+1}'0 
                    =  \W_{k+1}'\left( 0B_{4k-1} \W_{k+1}'\right)^{p-4}0\\
                &= \W_{k+1}' \U_{k+1}^{p-4}0  .
     \end{align*}
Also,     
       \begin{align*}
    0B_{4k+3} \W_{k+2}'&= 0 B_{4k+2}B_{4k+1}\W_{k+2}'     \\
                &= 0 \left( \W_{k+1}' \U_{k+1}^{p-4}0  \right)B_{4k+1}  \W'_{k+1} \U_{k+1}^{p-4} \left( \W_{k+1}\U_{k+1}^{p-3}\right)^{q-4}  \\
                &= \W_{k+1} \U_{k+1}^{p-4} \left( 0 B_{4k+1} \W'_{k+1} \right) \U_{k+1}^{p-4} \left( \W_{k+1}\U_{k+1}^{p-3}\right)^{q-4}  \\
                &= \W_{k+1} \U_{k+1}^{p-4} \left(\W_{k+1} \U_{k+1}\right) \U_{k+1}^{p-4} \left( \W_{k+1}\U_{k+1}^{p-3}\right)^{q-4} =  \U_{k+2}  .
     \end{align*}
Finally
    \begin{align*}
        B_{4k+4} &= B_{4k+3}^{q-4}B_{4k+2} = \left( B_{4k+2}B_{4k+1}\right)^{q-4} B_{4k+2}      \\
                &= \left( \W_{k+1}'\U_{k+1}^{p-4}0 B_{4k+1} \right)^{q-4} \W_{k+1}'\U_{k+1}^{p-4}0
            = \W_{k+1}'\U_{k+1}^{p-4} \left( 0 B_{4k+1} \W_{k+1}'\U_{k+1}^{p-4}\right)^{q-4}0 \\
                &= \W_{k+1}' \U_{k+1}^{p-4}\left( \W_{k+1}\U_{k+1}\U_{k+1}^{p-4}\right)^{q-4}0 
                    = \W_{k+2}'0
     \end{align*}
This finishes the induction. 

For the cases $\{p=4,q\ge5\}$, $\{p\ge 5,q=4\}$, $\{p=3,q\ge 7\}$, and $\{p\ge 7,q=3\}$, instead of Equations  
(\ref{eqn:SW1-pq})-(\ref{eqn:SW4-pq}), the following holds
\begin{align}
             0B_{2k-1} \W_{k+1}' = \U_{k+1} ,  \text{ if } q\neq 4, \quad   & \text{ or } \quad 
             \W_{k+1}0B_{2k-1} \W_{k+1}' = \U_{k+1}  ,  \text{ if } q=4  ,
        \label{eqn:SW1}   \\
      \, \, \,   B_{2k} & = \W_{k+1}' 0 .     \label{eqn:SW2} 
    \end{align}
A proof by induction of Equations (\ref{eqn:SW1}) and (\ref{eqn:SW2}) is analogous as before. 

We conclude, by Equations (\ref{eqn:SW4-pq}) and (\ref{eqn:SW2}),  that $\W(i)=B(i-1)$ for all $p,q$.
\end{proof}

\begin{proof}[Proof of Theorem \ref{thm:formula-W}]
From Lemma \ref{lem:Shallit}, for all $p$ and $q$ such that $(p-2)(q-2)>4$, we have that 
    \begin{align*}
        W(i) =   
        \begin{cases}
             4-B(i-1) , & \text{ if  } p=3, \\
             3-B(i-1) , & \text{ if  } p>3. \\
        \end{cases}
    \end{align*}
        From Lemma \ref{lem:beta}, we know that
    \begin{align*}
        \lfloor \beta \rfloor = 
        \begin{cases}
             q-4 , & \text{ if  } p=3, \\
             q-3 , & \text{ if  } p>3. \\
        \end{cases}
    \end{align*}
    Then, we get
        \begin{align*}
        W(i) &= q- \lfloor \beta \rfloor - B(i-1)   = q- \lfloor i\beta \rfloor + \lfloor (i-1) \beta \rfloor .
        \end{align*}
\end{proof}

\section{Proof of Lemma \ref{lem:P_k-n_k},  Lemma \ref{lem:m} and Lemma \ref{lem:bounds}}  \label{section:Pain}

\begin{proof}[Proof of Lemma \ref{lem:P_k-n_k}]
We compute $P_k - (p-2)(n_k-1)$ and factor by $\alpha^k$-terms and $\alpha^{-k}$-terms.
    \begin{align*}
   \big( P_k - (p-2)(n_k-1) \big) \left(\frac{\alpha-\alpha^{-1}}{p}\right) & =  \alpha^k+\alpha^{k-1} - (p-2)(q-2)\left(\frac{\alpha^k-1}{\alpha-1}\right)  \\
                &  \quad -  \alpha^{-k} -\alpha^{-(k-1)} + (p-2)(q-2)\left(\frac{\alpha^{-k}-1}{\alpha^{-1}-1} \right)    \\  
                & = \frac{\alpha^{k+1}+\alpha^{k}-\alpha^k-\alpha^{k-1} - (p-2)(q-2)(\alpha^k-1)}{\alpha-1}  \\
                & \quad  -  \frac{\alpha^{-(k+1)}+\alpha^{-k}-\alpha^{-k} -\alpha^{-(k-1)} - (p-2)(q-2)(\alpha^{-k}-1)}{\alpha^{-1}-1}  .
                \end{align*}
                
    Recall that $\alpha$ satisfies the cuadratic equation $\alpha^2 +2\alpha + 1 = (p-2)(q-2)\alpha $. It follows that  $\alpha^{k+1} +2\alpha^k + \alpha^{k-1} = (p-2)(q-2)a^k $ for $k \geq 1$. Thus, when we substitute this expression we obtain
    \begin{align*}
     \big( P_k -  (p-2)&(n_k-1) \big) \left(\frac{\alpha-\alpha^{-1}}{p}\right) = \\
    & =  \frac{\alpha^{k+1}-\alpha^{k-1} - (\alpha^{k+1}+2\alpha^k+\alpha^{-k}) +  (p-2)(q-2)}{\alpha-1}  \\
                & \quad  -  \frac{\alpha^{-(k+1)} -\alpha^{-(k-1)} - (\alpha^{-(k+1)}+2\alpha^{-k}+\alpha^{-(k-1)}) +  (p-2)(q-2)}{\alpha^{-1}-1}      \\ 
                & =  \frac{-2\alpha^{k}(1+\alpha^{-1}) + 2(1+\alpha^{-1})-2(1+\alpha^{-1}) +   (p-2)(q-2)}{\alpha-1}  \\
                & \quad  -  \frac{ -2\alpha^{-k}(1+\alpha^{-1}) + 2\alpha^{-k-1}-2\alpha^{-k+1} + 2(1+\alpha^{-1})-2(1+\alpha^{-1}) + (p-2)(q-2)}{\alpha^{-1}-1}      \\  
                & =  \frac{-2(\alpha^{k}-1)(1+\alpha^{-1}) }{\alpha-1} + \frac{ 2(\alpha^{-k}-1)(1+\alpha^{-1})}{\alpha^{-1}-1}
                  \\
                & \quad  +  \frac{-2(1+\alpha^{-1})+ (p-2)(q-2)}{\alpha-1} -   \frac{2\alpha^{-k}(\alpha^{-1}-\alpha)  -2(1+\alpha^{-1}) + (p-2)(q-2)}{\alpha^{-1}-1}  .      
    \end{align*}
    Now, $n_k - 1 = \frac{p(q-2)}{\alpha-\alpha^{-1}} \left( \frac{\alpha^k-1}{\alpha-1} - \frac{\alpha^{-k}-1}{\alpha^{-1}-1}  \right)  , $ and $\alpha+2+\alpha^{-1}= (p-2)(q-2)$, so
    \begin{align*}
    P_k - (p-2)(n_k-1)   
                & = \frac{-2(n_k-1)(1+\alpha^{-1})}{q-2} +                 \frac{p}{\alpha-\alpha^{-1}} \left(  \frac{\alpha-\alpha^{-1}}{\alpha-1} -   \frac{2\alpha^{-k}(\alpha^{-1}-\alpha) + \alpha-\alpha^{-1}}{\alpha^{-1}-1}   \right)     
    \end{align*}
Finally, we use that $\beta=\frac{q-2}{1+\alpha^{-1}}$ and rearrange terms to obtain 
    \begin{align*}
        P_k - (p-2)n_k & = \frac{-2(n_k-1)}{\beta} + \frac{p}{\alpha-1} - \frac{p}{\alpha^{-1}-1} - (p-2) +   \frac{2p\alpha^{-k}}{\alpha^{-1}-1}.   
    \end{align*}
This is equivalent to (\ref{eqn:P_k-n_k}).
\end{proof}

Let $\gamma_k=\frac{\alpha^k-\alpha^{-k}}{\alpha-\alpha^{-1}}$, then $\gamma_{k+1}+2\gamma_k+\gamma_{k-1}=(p-2)(q-2)\gamma_k$. Since $\gamma_0=0$ and $\gamma_1=1$, it follows that $\gamma_k$ is an integer for all $k\geq 1$. 

\begin{lem}
Let $U_k$ and $W_k$ be the words defined in Table \ref{table:UW}, and let $u_k$ and $w_k$ be their length, respectively. Then, 
	\begin{align} \label{eqn:length-uk}
    u_k =  
    \begin{cases}
        \gamma_k    ,      \\  
	    \gamma_k+\gamma_{k-1}  ,     \\
	    \gamma_k+\gamma_{k-1}  , 
    \end{cases}
     w_k =  
    \begin{cases}
        \gamma_k-\gamma_{k-1},    & \text{if $p=3, q\ge 7$ ,} \\
	    \gamma_k-(p-3)\gamma_{k-1} , & \text{if $p,q\ge 4$,} \\
	    \gamma_k-\gamma_{k-2} , & \text{if $p\ge7, q=3$.}
    \end{cases}
    \end{align}
	
\end{lem}
\begin{proof}
From the definition of the words $U_k$ and $W_k$, their lengths satisfy the following recurrence relations for the cases $p=3$, $p,q\ge 4$, and $q=3$, respectively.
	
	\begin{align} \label{eqn:recurrence-uk}
    u_{k+1} =  
    \begin{cases}
        (q-5)u_k+ w_k,      \\  
	   ((q-3)(p-3)+(p-4))u_k+(q-2)w_k ,     \\
	    (p-5)u_k+ w_k  , 
    \end{cases}
    \end{align}
    \begin{align} \label{eqn:recurrence-wk}
    w_{k+1} =
    \begin{cases}  
        (q-6)u_k+ w_k  ,  \\
	   ((q-4)(p-3)+(p-4))u_k+(q-2)w_k  ,  \\
	   (p-6)u_k+ w_k ,
    \end{cases}
    \end{align}	
Formulas (\ref{eqn:recurrence-uk}) and (\ref{eqn:recurrence-wk}) hold for $k=0$. Then an induction argument proves the result for any $k\geq 1$.

\end{proof}

\begin{proof}[Proof of Lemma \ref{lem:m}]
Denote by $W(i)$, $U_k(i)$ the $i$-th element of the words $W=\lim W_k$, $U_k$, respectively. Recall that,
    \begin{align*}
    W_{k+1} = 
        \begin{cases}
        W_kU_k^{q-6}, & \text{if $p=3, q\ge 7$}      \\
        W_k U_k^{p-4} \left( W_k U_k^{p-3} \right)^{q-4} , & \text{if $p,q\ge 4$}      \\
	    W_k U_k^{p-6} , & \text{if $p\ge 7, q=3$}.      
        \end{cases}
    \end{align*}
    Combining this with formula  $ W(i)= q - \lfloor i \beta \rfloor + \lfloor (i-1) \beta \rfloor $ of Theorem \ref{thm:formula-W} we obtain, 
    \begin{align} \label{eqn:formulaU_k,i}
    U_k(i) = 
        \begin{cases}
        W(i+w_k) = q - \lfloor (i+w_k) \beta \rfloor + \lfloor (i-1+w_k) \beta \rfloor ,  & \text{if $p=3$ or $p> 4$}      \\
        W(i+2w_k) = q - \lfloor (i+2w_k) \beta \rfloor + \lfloor (i-1+2w_k)  \beta \rfloor,  & \text{if $p =  4$}      \\
        \end{cases}
    \end{align}
Now, $d_k=(U_k)^p$, if $p>3, k\ge 1$; and $d_k=(U_k U_{k-1})^3$ if $p=3, k\ge 2$. In either case, $d_k$ consists of $p$ identical blocks, each of length $\gamma_k+\gamma_{k-1}$. 

Recall that $m=m(n)$ satisfies: 
    \begin{align*} 
        \sum_{i=1}^m (q-d_{k}(i)) \leq n- n_k -1 < \sum_{i=1}^{m+1} (q-d_{k}(i)).        
    \end{align*}

For all $p$, the expression $p \cdot \sum_{i=1}^{\gamma_k+\gamma_{k-1}} (q-d_{k,i})$ counts number of tiles in the $k+1$ layer of $A_{p,q}(k+1)$. That is, $\sum_{i=1}^{\gamma_k+\gamma_{k-1}} (q-d_{k,i}) = l_k$, where 
    \begin{align*}
    l_k & =\frac{n_{k+1}-n_k}{p}    \\
        & = \frac{q-2}{\alpha-\alpha^{-1}} \left( \frac{\alpha^{k+1}-1}{\alpha-1} -\frac{\alpha^{-(k+1)}-1}{\alpha^{-1}-1} -\frac{\alpha^k-1}{\alpha-1} + \frac{\alpha^{-k}-1}{\alpha^{-1}-1}\right) \\
        & = \frac{q-2}{\alpha-\alpha^{-1}} \left( \frac{\alpha^{k+1}-\alpha^k}{\alpha-1} -\frac{\alpha^{-(k+1)}-\alpha^{-k}}{\alpha^{-1}-1} \right) \\
        & = (q-2) \gamma_k = \frac{\gamma_{k+1}+2\gamma_k+\gamma_{k-1}}{p-2}
    \end{align*}
    
If $p=3$, we have $\beta = \alpha+1$, $u_k=\gamma_k$, $w_k=\gamma_k-\gamma_{k-1}$. Note that $\gamma_{k+1}-\gamma_k \alpha = \alpha^{-k}$ and $w_k\alpha-w_{k+1}=\alpha^{-k}(\alpha-1)$. 
Thus, if $p=3$, 
    \begin{align*}
    \sum_{i=1}^{\gamma_k} (q-d_{k}(i)) &= \sum_{i=1}^{\gamma_k} (q-U_k(i))
             =  \lfloor (u_k+w_k) \beta \rfloor - \lfloor w_k \beta \rfloor     \\
            &=  \lfloor u_{k+1}+w_{k+1} +u_k+w_k + \alpha^{-k}(\alpha-2) \beta \rfloor - \lfloor w_{k+1}+w_k + \alpha^{-k}(\alpha-1) \rfloor \\
            &= u_{k+1}+u_k = \gamma_{k+1}+\gamma_k = \frac{l_{k+1}+l_k}{q-2}.
    \end{align*}
Summarizing      
    \begin{align} 
        0 &\le m \le \gamma_k+\gamma_{k-1} && \Longleftrightarrow &&  1 \le n-n_k \le  l_k &&\text{ if } p>3, \label{eqn:range-uk1}\\
        0 &\le m \le \gamma_k &&\Longleftrightarrow&&  1 \le n-n_k \le  \gamma_{k+1}+\gamma_k &&\text{ if } p=3. \label{eqn:range-uk2}
    \end{align}     
     
We now deduce a formula for $m=m(n)$, when $m$ and $n$ are within the values given by (\ref{eqn:range-uk1}) and (\ref{eqn:range-uk2}). In this range, $d_k(i)=U_k(i)$ for $i\le m$, and we can write the condition (\ref{eqn:formula-m}) for $m=m(n)$ as:
    \begin{align*} 
        \sum_{i=1}^m (q-U_{k}(i)) & \leq n- n_k -1 < \sum_{i=1}^{m+1} (q-U_{k}(i)) , \\
        \lfloor (m+\delta w_k) \beta \rfloor - \lfloor \delta w_k \beta \rfloor  & \leq n-n_k-1 
            <   \lfloor (m+1+\delta w_k) \beta \rfloor - \lfloor\delta w_k \beta \rfloor.
    \end{align*}
Where $\delta = 1$ if $p=3$ or $p> 4$, and $\delta =2$ if $p=4$.    
 The left-hand side implies,
    \begin{align*}
         m\beta + \delta w_k \beta - 1 -  \lfloor\delta w_k \beta \rfloor & < n-n_k-1 , \\
        m & < \frac{n-n_k- \{\delta w_k \beta \}}{\beta} . 
    \end{align*}
Where $\{\delta w_k \beta \} = \delta w_k\beta -  \lfloor\delta w_k \beta \rfloor$. Similarly, the right-hand side implies,     
    \begin{align*}
     \frac{n-n_k-\{\delta w_k \beta \} -\beta}{\beta} < m. 
    \end{align*}
Since the difference between these bounds is 1, we arrive at the following formula:
    \begin{align} \label{eqn:formula-m2}
        m(n)= \left \lfloor  \frac{n-n_k-\{\delta w_k \beta \}}{\beta} \right \rfloor, \text{ for } 
        \begin{array}{l}
              1\le n-n_k \le l_k, \qquad \quad \text{ if } p>3,    \\
              1\le n-n_k \le \gamma_{k+1}+\gamma_k \, \text{ if } p=3.
        \end{array}
    \end{align}
The explicit value of $\{\delta w_k \beta \} = \delta w_k\beta -  \lfloor\delta w_k \beta \rfloor$ can be computed using  $\gamma_{k+1} - \gamma_k \alpha  = \alpha^{-k}$ and $\beta = \frac{\alpha+1}{p-2}$, obtaining:
   \begin{align*} 
   \{\delta w_k \beta \} =  \begin{cases}
    \alpha^{-k}(\alpha-1)   & \text{if $p=3$ or $4$},\\    
    \alpha^{-k}(\alpha-\beta)   & \text{if $p>4$}.  
  \end{cases}
    \end{align*}

Equation (\ref{eqn:formula-m2}) is the key ingredient for Lemma \ref{lem:m}. It is only left to use the structure of $d_k$ in $p$ blocks. When $p>3$, we have $d_k=U_k^p$, and the length of $U_k$ is $\gamma_k+\gamma_{k-1}$. Now, for any $1\le n-n_k < p\cdot l_k$, 
let 
    \begin{align*}
        j &= \left \lfloor \frac{n-n_k-1 }{l_k} \right \rfloor, \\
        n' &= n-n_k - j \cdot l_k.
    \end{align*}
Then $0\le j\le p-1$ is the position of the block of $d_k$ to which $m(n)$ belongs, and  $n'$ is the number of the position of the  $(n-n_k)$-th tile within that block. 

Let $m'=m(n')$. If $p>3$ we can apply Equation (\ref{eqn:formula-m2}). If $p=3$, then $d_k=(U_kU_{k-1})^3$, and we have to consider whether $m'$ falls  in a $U_k$-block or a $U_{k-1}$-block. Taking this into consideration we arrive at:
   \begin{align*}
        m' = 
        \begin{cases}
            \vspace{.1cm}
            \left \lfloor \frac{n' -  \alpha^{-k}(\alpha-\beta) }{\beta} \right \rfloor, & \text{ if } p>4 , \\
            \vspace{.1cm}
              \left \lfloor \frac{n' - \alpha^{-k}(\alpha-1)  }{\beta} \right \rfloor, & \text{ if } p=3 \text{ and } n' \le \gamma_{k+1}+\gamma_k \text{ or if } p=4, \\
              \left \lfloor \frac{n' -(\gamma_{k+1}+\gamma_k)  - \{ w_{k-1} \beta \} }{\beta} \right \rfloor + \gamma_k, 
              & \text{ if }   p=3 \text{ and }  \gamma_{k+1}+\gamma_k < n' \le l_k.
        \end{cases}
    \end{align*}
 
Finally, $m(n) = m'+ j ( \gamma_k+\gamma_{k-1})$, for all $1\le n-n_k \le p\cdot  l_k$. Now we simplify the expression for $m'$ for the last case. For $p=3$, we have $\{ w_{k-1} \beta \}=\alpha^{-(k-1)}(\alpha-1)$, so:
    \begin{align*}
        \left \lfloor \frac{n' -(\gamma_{k+1}+\gamma_k)  - \{ w_{k-1} \beta \} }{\beta} \right \rfloor + \gamma_k 
            &= \left \lfloor \frac{n' -(\gamma_{k+1}+\gamma_k)  - \alpha^{-(k-1)} (\alpha-1) +\beta \gamma_k}{\beta} \right \rfloor 
    \end{align*}    
We use $\beta=\alpha+1$ (valid for $p=3$) and $\alpha \gamma_k= \gamma_{k+1}-\alpha^{-k}$:
    \begin{align*}
    \left \lfloor \frac{n' -(\gamma_{k+1}+\gamma_k)  - \{ w_{k-1} \beta \} }{\beta} \right \rfloor + \gamma_k  &= \left \lfloor \frac{n' -(\gamma_{k+1}+\gamma_k)  - \alpha^{-(k-1)} (\alpha-1) +\gamma_{k+1} - \alpha^{-k} +  \gamma_k}{\beta} \right \rfloor \\
            &= \left \lfloor \frac{n' - \alpha^{-k}(\alpha (\alpha-1)+1) }{\beta} \right \rfloor 
    \end{align*}    
Summarizing, for any $p$, we can write $m$ as:
    \begin{align*}
        m &= m' + j(\gamma_k+\gamma_{k-1}) = \left \lfloor \frac{n' - \phi_k }{\beta} \right \rfloor +  j(\gamma_k+\gamma_{k-1}),
    \end{align*}
where $\phi_k$ is defined as in (\ref{eqn:phi}).

Now, we simplify the formula for $m$ using the fact that $\frac{l_k}{\beta}- (\gamma_k+\gamma_{k-1})=\alpha^{-k}$.
    \begin{align*}  
    m &  = m' +  j ( \gamma_k+\gamma_{k-1})  ,  \\
        &= \left \lfloor \frac{  n-n_k - \left \lfloor \frac{n-n_k-1 }{l_k} \right \rfloor \cdot l_k - \phi_k }{\beta} \right \rfloor + \left \lfloor \frac{n-n_k-1 }{l_k} \right \rfloor ( \gamma_k+\gamma_{k-1})  \\
        &= \left \lfloor \frac{  n-n_k - \phi_k}{\beta} + \left \lfloor \frac{n-n_k-1 }{l_k} \right \rfloor ( \gamma_k+\gamma_{k-1}) -  \left \lfloor \frac{n-n_k-1 }{l_k} \right \rfloor \frac{ l_k  }{\beta} \right \rfloor  \\
        &= 
    \left \lfloor \frac{n-n_k - \phi_k }{\beta}   - \left \lfloor \frac{n-n_k-1 }{l_k} \right \rfloor \cdot \alpha^{-k} \right \rfloor .
    \end{align*}
Finally, when $n=n_{k+1}$ there is a deficit of 1 in Equation (\ref{eqn:formula-m}), it gives the value of $m= p(\gamma_k+\gamma_{k-1})-1$. This comes from the fact that, when we attach a new tile, we are counting the vertex trapped between the $k$th layer and the previously attached tiles. But when we attach the $n=n_{k+1}$ tile, we are closing the $(k+1)$-layer, so we surround two vertices as opposed to just one. To correct for this special case, we need to add the term $ \left \lfloor \frac{n- n_k }{n_{k+1}-n_k} \right \rfloor $ to the previous formula, thus arriving at equation (\ref{eqn:m}).


\end{proof}

Finally, we prove the bounds that we used in the proof of Lemma \ref{lem:bounds}.

\begin{proof} [Proof of Lemma \ref{lem:bounds}]
We say that a function $F=F(p,q)$ is increasing in $p,q$ if $F$ is increasing in $p$ while $q$ remains constant, and if $F$ is increasing in $q$ while $p$ remains constant. The same convention is adopted for decreasing in $p,q$. We use this concept to bound functions by its values at the cases $\{p,q\}=\{3,7\}, \{4,5\}, \{5,4\}, \{7,3\}$.

For example, $t=(p-2)(q-2)-2$ is increasing in $p,q$, then so is $\alpha=\frac{t+\sqrt{t^2-4}}{2}$. On the other hand, $\alpha^{-1}$ is decreasing in $p,q$, then so is $\frac{1}{\beta} = \frac{\alpha^{-1}+1}{q-2}$. 
Therefore $\frac{1}{\beta} $ is bounded by its biggest value on the cases $\{p,q\}=\{3,7\}, \{4,5\}, \{5,4\}, \{7,3\}$. Then it can be verified that $\frac{1}{\beta}<1.4$ for all $p,q$. Similarly $\frac{1}{\alpha}<0.4$

Next,  $\frac{t}{p} =\frac{pq-2p-2q+2}{p} = q-2-\frac{q-2}{p} $ is increasing in $p,q$. Similarly $\frac{t+1}{p}$, $\frac{t-1}{p}$ and $\frac{t-2}{p}$ are increasing in $p,q$. 
By composition of increasing functions, we obtain that 
    \begin{align*}
        \frac{\alpha}{p} = \frac{1}{2} \left( \frac{t}{p}+\sqrt{\left( \frac{t}{p}\right)^4 - \frac{4}{p}}  \right) 
    \end{align*}
is also increasing in $p,q$. Then, $\frac{p}{\alpha}$ is decreasing in $p,q$. By evaluating at the cases $\{p,q\}=\{3,7\}, \{4,5\}, \{5,4\}, \{7,3\}$, we find that $\frac{p}{\alpha}< 2.7$.

Similarly
    \begin{align*}
        \frac{\alpha-1}{p} = \frac{1}{2} \left( \frac{t-2}{p}+\sqrt{\left( \frac{t}{p}\right)^4 - \frac{4}{p}}  \right) 
    \end{align*}
is increasing in $p,q$. So $\frac{p}{\alpha-1}$ is decreasing in $p,q$. We find that   $\frac{p}{\alpha-1}<4.4$.

Now, we bound $\frac{\phi_2}{\beta}$ for all its cases:
    \begin{align*}
    &    \frac{\alpha-1}{\alpha^{2}\beta}  < \frac{1}{\alpha \beta } <\frac{1}{\beta} < 1.4,     \\
    &   \frac{\alpha(\alpha-1)+1}{\alpha^2 \beta } = \frac{\alpha-1}{\alpha \beta}+ \frac{1}{\alpha^2 \beta} < \frac{2}{\beta}<2.8, \\
    &    \frac{\alpha-\beta}{\alpha\beta} =\frac{1}{\beta}-\frac{1}{\alpha} = \frac{p+2}{\alpha+1}-\frac{1}{\alpha} = \frac{(p-3)\alpha-1}{\alpha(\alpha+1)}  < \frac{(p-3)\alpha}{\alpha(\alpha+1)}  
            < \frac{p}{\alpha+1} < \frac{p}{t} < 2.4 .
    \end{align*}
In conclusion,     $\frac{\phi_2}{\beta}< 2.8$ for all $p,q$. Finally, 
    \begin{align*}
        \frac{p}{\alpha^2} = \frac{p}{\alpha} \cdot \frac{1}{\alpha} < 2.7 \cdot 0.4 < 1.1 .
    \end{align*}
\end{proof}

  

\subsection*{Acknowledgements}
This project received funding from the European Union's Horizon 2020 research and innovation program under the Marie Sk\l odowska-Curie grant agreement No.~754462. This research has also been supported by the DFG Collaborative Research Center SFB/TRR 109 \textit{Discretization in Geometry and Dynamics}. The first author thanks the Laboratory for Topology and Neuroscience for hosting them during part of the project.

\bibliographystyle{amsplain}
\bibliography{holeyominoes}

\providecommand{\bysame}{\leavevmode\hbox to3em{\hrulefill}\thinspace}
\providecommand{\MR}{\relax\ifhmode\unskip\space\fi MR }
\providecommand{\MRhref}[2]{%
  \href{http://www.ams.org/mathscinet-getitem?mr=#1}{#2}
}
\providecommand{\href}[2]{#2}
\begin{thebibliography}{1}

\bibitem{harary1976extremal}
Frank Harary and Heiko Harborth, \emph{Extremal animals}, J. Combinatorics
  Information Syst. Sci. \textbf{1} (1976), no.~1, 1--8. \MR{0457263}

\bibitem{MaRoTo}
Greg Malen, {\'E}rika Rold{\'a}n, and Rosemberg Toal{\'a}-Enr{\'\i}quez,
  \emph{Extremal $\{p, q\}$-animals}, arXiv preprint arXiv:2109.05331 (2021).

\bibitem{shallit1991characteristic}
Jeffrey Shallit, \emph{Characteristic words as fixed points of homomorphisms},
  University of Waterloo. Department of Computer Science, 1991.

\end{thebibliography}

\end{document}